\documentclass[reqno,a4paper]{amsart}
\usepackage[utf8]{inputenc}
\usepackage[T1]{fontenc}
\usepackage{amssymb,amsthm,amsmath}
\usepackage{calrsfs}
\usepackage{hyperref}
\usepackage{mathbbol}
\usepackage{mathabx}
\usepackage[all]{xy}
\usepackage{xcolor}

\newdir{>}{{}*:(1,-.2)@^{>}*:(1,+.2)@_{>}}
\newdir{<}{{}*:(1,+.2)@^{<}*:(1,-.2)@_{<}}
 
\theoremstyle{plain}
\newtheorem{theorem}[subsection]{Theorem}
\newtheorem{lemma}[subsection]{Lemma}

\theoremstyle{remark}
\newtheorem{remark}[subsection]{Remark}

\allowdisplaybreaks

\newcommand{\noproof}{\hfill \qed}
\newcommand{\comp}{\circ}
\newcommand{\defn}{\textbf}
\newcommand{\To}{\Rightarrow}

\newcommand{\K}{\ensuremath{\mathbb{K}}}
\newcommand{\X}{\ensuremath{\mathbb{X}}}
\newcommand{\Z}{\ensuremath{\mathbb{Z}}}

\newcommand{\Ab}{\ensuremath{\mathsf{Ab}}}
\newcommand{\Mod}{\ensuremath{\mathsf{Mod}}}

\newcommand{\ExtS}{\ensuremath{\mathsf{ExtS}}}
\newcommand{\Ext}{\ensuremath{\mathsf{Ext}}}
\newcommand{\Vect}{\ensuremath{\mathsf{Vect}}}
\newcommand{\Set}{\ensuremath{\mathsf{Set}}}

\DeclareMathOperator{\Ker}{Ker}

\newcommand{\LACC}{{\rm (LACC)}}

\begin{document}

\title[A universal Kaluzhnin--Krasner embedding theorem]{A universal Kaluzhnin--Krasner\\ embedding theorem}

\author{B.~S.~Deval}
\author{X.~García-Martínez}
\author{T.~Van der Linden}

\email{bo.deval@uclouvain.be}
\email{xabier.garcia.martinez@uvigo.gal}
\email{tim.vanderlinden@uclouvain.be}

\address[Xabier García-Martínez]{CITMAga \& Universidade de Vigo, Departamento de Ma\-temáticas, Esc.\ Sup.\ de Enx.\ Informática, Campus de Ourense, E--32004 Ourense, Spain}
\address[Bo Shan Deval, Tim Van der Linden]{Institut de Recherche en Mathématique et Physique, Université catholique de Louvain, chemin du cyclotron 2 bte L7.01.02, B--1348 Louvain-la-Neuve, Belgium}
\address[Tim Van der Linden]{Mathematics \& Data Science, Vrije Universiteit Brussel, Pleinlaan 2, B--1050 Brussel, Belgium}

\thanks{The first author's research is supported by a grant of the Fund for Research Training in Industry and Agriculture (FRIA). The second author is supported by Ministerio de Ciencia e Innovación (Spain), with grant number PID2021-127075NA-I00. The third author is a Senior Research Associate of the Fonds de la Recherche Scientifique--FNRS}

\begin{abstract}
	Given two groups $A$ and $B$, the \emph{Kaluzhnin--Krasner universal embedding theorem} states that the wreath product $A\wr B$ acts as a universal receptacle for extensions from $A$ to $B$. For a split extension, this embedding is compatible with the canonical splitting of the wreath product, which is further universal in a precise sense. This result was recently extended to Lie algebras and to cocommutative Hopf algebras.
	
	The aim of the present article is to explore the feasibility of adapting the theorem to other types of algebraic structures. By explaining the underlying unity of the three known cases, our analysis gives necessary and sufficient conditions for this to happen.
	
	From those we may for instance conclude that a version for crossed modules can indeed be attained, while the theorem cannot be adapted to, say, associative algebras, Jordan algebras or Leibniz algebras, when working over an infinite field: we prove that then, amongst non-associative algebras, only Lie algebras admit a universal Kaluzhnin--Krasner embedding theorem.
\end{abstract}


\maketitle


\section{Introduction}
Given two groups $A$ and $B$, the \emph{Kaluzhnin--Krasner universal embedding theorem}~\cite{KKThm} says that the wreath product $A\wr B$ acts as a universal receptacle for any group $G$ viewed as an extension from $A$ to $B$. More precisely,

\begin{theorem}\label{Theorem KK}
	For any group extension $0 \to A \to G \to B \to 0$, the group $G$ can be embedded into the wreath product via a group homomorphism $\phi_G\colon G \to A \wr B$.\noproof
\end{theorem}

Recall that the wreath product $A \wr B$ is the group $\Set(B, A) \rtimes B$, where the group structure on the set of functions $\Set(B, A)$ from $B$ to $A$ is pointwise (given $h$, $h'\colon B\to A$ and $b\in B$ we put $(hh')(b)=h(b)h'(b)$) and the action of $B$ on $\Set(B, A)$ is canonical ($h(b)^{b'}=h(bb')$ for all $b$, $b'\in B$). Fix a set-theoretical section $s\colon B\to G$ of $f\colon G\to B$. The homomorphism $\phi_G\colon G \to A \wr B$ takes any $g\in G$ and sends it to the couple $(h_g, f(g))\in A\wr B$ where $h_g\colon B\to A\colon b\mapsto s(b)\cdot g\cdot s(b\cdot f(g))^{-1}$. It is not hard to check by hand that this is indeed an injection.

This embedding has some convenient properties, making it an important tool in group theory. Let us just mention here that, by definition, the wreath product $A\wr B$ induces a split extension
\begin{equation}\label{EqWreathSE}
	\xymatrix{
		0 \ar[r] & \Set(B, A) \ar[r]_-\kappa & \Set(B, A) \rtimes B \ar@<.5ex>[r]^-\pi & B \ar@<.5ex>[l]^-\sigma \ar[r] & 0\text{,}
	}
\end{equation}
which allows us to see Theorem~\ref{Theorem KK} as part of an embedding of extensions. Any given extension $E=(0 \to A \to G \to B \to 0)$ from $A$ to $B$ embeds into the wreath product (split) extension, via a monomorphism $\phi$ in the category of group extensions as in
\[
	\xymatrix{
	0 \ar[r] & A \ar[d]_-{\phi_A} \ar[r]^-k & G \ar[d]_-{\phi_G} \ar[r]^-f & B \ar@{=}[d]\ar[r] & 0\\
	0 \ar[r] & \Set(B, A) \ar[r]_-\kappa & A\wr B \ar[r]_-\pi & B \ar[r] & 0\text{.}
	}
\]
Here $\phi_A$ is induced by the universal property of the kernel $\kappa$; since the composite $\phi_G\comp k$ is a monomorphism, so is the morphism $\phi_A$. The diagram forms a monomorphism of group extensions, because it is a monomorphism in the category of group extensions over $B$, where monomorphisms are pointwise. This justifies our use of the word ``embedding''.

As we explained above, the morphism $\phi$ depends on the choice of a set-theoretical section $s\colon B\to G$ of the surjection $f$. Whenever this $s$ is a group homomorphism, there is an action of $B$ on $A$ induced by $s$ through the equivalence between split extensions and group actions based on the semidirect product construction. Likewise, the section $\sigma\colon B\to A\wr B$ of the wreath product split extension corresponds to the canonical action of $B$ on $\Set(B,A)$. The morphism $\phi$ is equivariant with respect to these actions: for each $a\in A$ and $b\in B$ we have that $\phi_A(a)^b=\phi_A(a^b)$, because
\begin{align*}
	\phi_A(a)^b(b') & = h_a(b'b) = s(b'b)a s(b'b)^{-1}= s(b')s(b)a s(b)^{-1}s(b')^{-1} \\
	                & =h_{s(b)a s(b)^{-1}}(b')= \phi_A(a^b)(b')\text{,}
\end{align*}
for all $b'\in B$. This is equivalent to saying that $\phi$ is a morphism of \emph{split} extensions ($\phi_G\circ s=\sigma$). The wreath product $A\wr B$ further satisfies a strong type of universality which we will study in detail in this article.

Recently there has been some effort towards extending this embedding to other algebraic settings: \cite{PRS} considers a Kaluzhnin--Krasner embedding theorem for Lie algebras, while \cite{BST} obtains a theorem in the context of cocommutative Hopf algebras. In both cases, the key problem is of course to determine how the wreath product $A\wr B$ should be defined in the given context. The aim of the present article is to expose the underlying unity of these different results, and explain that there is a general recipe for the wreath product---a \emph{universal Kaluzhnin--Krasner embedding theorem}---solving this problem once and for all, for any type of algebraic structure whatsoever, under some precise conditions, which allows us to make predictions about the feasibility of further extending the result to other contexts.

\section{The case of split extensions}
We start by considering a special case of the embedding theorem: its restriction to split extensions, which turns out to be at the heart of the problem, because it gives us a formula for $A\wr B$.

\subsection{Universality of the wreath product}\label{Section Construction R}
Let $B$ be a fixed group. We write $R(A)$ for the wreath product split extension \eqref{EqWreathSE}, omitting the $B$ from the notation because it is assumed fixed. We write $R$ for the functor from groups to split extensions over $B$ that sends a morphism $\gamma\colon A\to C$ to the morphism of split extensions $R(\gamma)\colon R(A)\to R(C)$ determined by
\begin{equation}\label{Eq R Natural}
	\gamma \wr B\colon A\wr B\to C\wr B\colon (h\colon B\to A,b)\mapsto (\gamma\comp h\colon B\to C,b)\text{.}
\end{equation}

Given a split extension
\begin{equation}\label{Eq Split Extension S}
	S=\bigl(\xymatrix{
		0 \ar[r] & A \ar[r]^-k & G \ar@<.5ex>[r]^-f & B \ar@<.5ex>[l]^-s \ar[r] & 0
	}\bigr)
\end{equation}
over~$B$, we shall denote its inclusion into the wreath product split extension $R(A)$ by $\eta_S\coloneq \phi\colon S\to R(A)$. Now we can make the following observation, first in the context of groups. Suppose we have another wreath product $C\wr B$ over~$B$, together with a morphism of split extensions  $\alpha\colon S\to R(C)$ as in
\begin{equation}\label{Eq Diagram Universality}
	\vcenter{\xymatrix{S \ar[r]^-{\eta_S} \ar[rd]_-{\forall\alpha} & R(A) \ar@{-->}[d]^-{R(\overline{\alpha})} & A\ar@{-->}[d]^-{\exists !\overline{\alpha}}\\
	& R(C) & C}}
\end{equation}
then there exists a unique morphism $\overline{\alpha}\colon A\to C$ such that $R(\overline{\alpha})\circ\eta_S=\alpha$. We may take $\overline{\alpha}(a)=K(\alpha)(a)(1)$ where $1$ is the neutral element of the group~$B$ and the morphism \(K(\alpha)\colon A\to \Set(B,C)\) is the component between the kernel objects of the map \(\alpha \colon S \to R(C)\). Indeed
\begin{align*}
	K(R(\overline\alpha)\comp\eta_S)(a) & = KR(\overline\alpha)\bigl(h_{k(a)}\colon B\to A\colon b\mapsto s(b)k(a)s(b)^{-1}\bigr)         \\
	                                    & = \overline\alpha \comp h_{k(a)}\colon B\to C\colon b\mapsto \overline\alpha(s(b)k(a)s(b)^{-1}) \\
	                                    & = \overline\alpha \comp h_{k(a)}\colon B\to C\colon b\mapsto K(\alpha)(s(b)k(a)s(b)^{-1})(1)    \\
	                                    & = \overline\alpha \comp h_{k(a)}\colon B\to C\colon b\mapsto K(\alpha)(a^{b})(1)                \\
	                                    & = \overline\alpha \comp h_{k(a)}\colon B\to C\colon b\mapsto (K(\alpha)(a))^{b}(1)              \\
	                                    & = \overline\alpha \comp h_{k(a)}\colon B\to C\colon b\mapsto K(\alpha)(a)(1b)                   \\
	                                    & = \overline\alpha \comp h_{k(a)}\colon B\to C\colon b\mapsto K(\alpha)(a)(b)                    \\
	                                    & = K(\alpha)(a)\colon B\to C\colon b\mapsto K(\alpha)(a)(b)\text{;}
\end{align*}
note that this verification on the kernel $A$ of the extension $S$ suffices for the triangle in \eqref{Eq Diagram Universality} to commute. On the other hand, if $K(R(\widetilde\alpha)\comp \eta_S) = K(\alpha)$ for some $\widetilde{\alpha}\colon A\to C$, then the first steps in the above calculation show that $K(\alpha)(a)=\widetilde\alpha\comp h_{k(a)}$ which, when we evaluate in $b=1$, shows that $\widetilde{\alpha}=\overline{\alpha}$ and proves uniqueness of our chosen~$\overline{\alpha}$.

Note that we make this explicit for groups, but it is not hard to see that the same holds for Lie algebras or cocommutative Hopf algebras, each with their respective universal embeddings.

This ``universality of the wreath product'' just means that $\eta_S$ is the $S$-component of the unit of an adjunction $K\dashv R$, where $K$ is the forgetful functor sending a split extension $S$ over $B$ as in \eqref{Eq Split Extension S} to the kernel object $A$, and $R$ is the functor sending an object $A$ to the wreath product split extension over $B$ determined by $A\wr B$ as explained above. Now the only thing missing for a general result unifying the cases of groups, Lie algebras and cocommutative Hopf algebras is a precise description of a context where such an adjunction (1) makes sense and (2) exists.

\subsection{A context for the general analysis}\label{Section Context}
An appropriate context is the setting of \emph{semi-abelian categories} in the sense of Janelidze--Márki--Tholen~\cite{Janelidze-Marki-Tholen}: here the functor $K$ plays a special role, for instance in the study of semidirect products and internal actions~\cite{Bourn-Janelidze:Semidirect}, and the adjunction $K\dashv R$ is already investigated in detail in the literature~\cite{Gray2012}. This context includes the three types of algebraic structures as examples (see~\cite{GKV2} for the Hopf algebra case), and more generally any type of (universal) algebras containing a group operation and where the neutral element of
the group structure forms a subalgebra (the \emph{varieties of $\Omega$-groups} of Higgins~\cite{Higgins}), as well as other interesting but less classical categories such as loops, Heyting semilattices and the dual of the category of pointed sets (see~\cite{Borceux-Bourn, acc,Rodelo:Moore} for an overview).

By definition, a category $\X$ is \emph{semi-abelian} if and only if it is pointed, Barr exact~\cite{Barr} and Bourn protomodular~\cite{Borceux-Bourn,Bourn1991} with binary coproducts. \emph{Pointed} means that there is a \emph{zero object}: an initial object $0$ which is also terminal. Recall that for a category to be \emph{abelian}, one needs to add \emph{additivity} to Barr exactness: the existence of a natural abelian group structure on the hom-sets, which is not available for groups, Lie algebras or Hopf algebras. The protomodularity condition which replaces it says that the \emph{Split Short Five Lemma} holds in $\X$, or equivalently, that the middle object $G$ in a
diagram such as
\[
	\xymatrix{
		A \ar[r]^-k & G \ar@<.5ex>[r]^-f & B \ar@<.5ex>[l]^-s
	}
\]
where $f\comp s=1_B$ and $k$ is a kernel of $f$ is ``covered'' or ``generated'' by the outer objects $A$ and $B$ in the sense that they do not both factor through the same proper subobject of $G$. This implies that $f$ is a cokernel of $k$, so that the diagram forms a split extension.

\subsection{Local algebraic cartesian closedness}
We fix a semi-abelian category $\X$, and write $\ExtS(B)$ for the category of split extensions over $B$ in $\X$: a morphism $\phi=(\phi_A,\phi_G)$ from $S=(A,G,k,f,s)$ such as \eqref{Eq Split Extension S} to $T=(C,H,l,g,t)$ satisfies $\phi_G\circ k=l\circ \phi_A$, $g\circ \phi_G=f$ and $\phi_G\circ s=t$.
\[
	\xymatrix{
	0 \ar[r] & A \ar[d]_-{\phi_A} \ar[r]^-k & G \ar[d]_-{\phi_G} \ar@<.5ex>[r]^-f & B \ar@<.5ex>[l]^-s \ar@{=}[d]\ar[r] & 0\\
	0 \ar[r] & C \ar[r]_-l & H \ar@<.5ex>[r]^-g & B \ar@<.5ex>[l]^-t \ar[r] & 0
	}
\]
We write $K\colon \ExtS(B)\to \X$ for the forgetful functor which sends a split extension $S=(A,G,k,f,s)$ to the kernel $A$, so that $A=K(S)$.

The existence of a right adjoint $R\colon \X\to \ExtS(B)$ to $K$ was first considered by James R.\ A.\ Gray in his Ph.D.\ thesis \cite{GrayPhD} and further studied in the articles \cite{Bourn-Gray,GrayLie,Gray2012}, amongst others. The description of $R(A)$ for groups in \cite{Gray2012} is precisely the classical wreath product, while $R(A)$ for Lie algebras described in \cite{GM-G, GrayLie} coincides with the wreath product of \cite{PRS}.

We did not know of an explicit description in the literature of $R(A)$ in the case of Hopf algebras, even though by~\cite[Proposition~5.3]{Gray2012} the existence of the functor $R$ was already clear from the fact that the category of cocommutative Hopf algebras over a field~$\K$ is the category of internal groups in the category of cocommutative coalgebras over~$\K$, which is known to be cartesian closed and finitely complete~\cite{Barr-Coalgebras,GM-VdL1}. Since, as explained in Subsection~\ref{Section Construction R}, the wreath product of~\cite{BST} provides us with a construction of the functor~$R$, we now do have an explicit description of the right adjoint of the forgetful functor $K$.

In fact, the existence of a right adjoint $R\colon \X\to \ExtS(B)$ does not come for free, and just a few types of semi-abelian categories are known where it does exist for all $B$. If so\footnote{Our main focus here is on categories where the right adjoint exists for \emph{all} objects $B$. We might, however, consider the condition just for one fixed $B$, which leads to ``local'' versions of results such as Theorem~\ref{Theorem Embedding Split}, applicable in more general settings.}, then~$\X$ is said to be \emph{locally algebraically cartesian closed (LACC)}. The relative strength of the condition is witnessed by the fact that the only examples of (LACC) categories currently known in the literature are: essentially affine categories (which include all additive categories)~\cite{Bourn1991}; internal groups in a cartesian closed category with pullbacks (which include the examples of classical groups, crossed modules and cocommutative Hopf algebras)~\cite{Bourn-Gray}; and internal Lie algebras in an additive cocomplete symmetric closed monoidal category~\cite{GM-G}. On the other hand, almost any other known semi-abelian category is known \emph{not} to be (LACC)---see below (Theorem~\ref{Theorem Algebras}) for a concrete result in the context of algebras over a field.

\subsection{Towards an embedding theorem for split extensions}\label{Section Construction Wreath}
Let us now reason within a chosen semi-abelian category $\X$. Let~$B$ be an object of $\X$. The existence of a right adjoint $R$ for the forgetful functor $K\colon \ExtS(B)\to \X$ implies that each split extension $S$ as in \eqref{Eq Split Extension S} comes equipped with a morphism ${\eta_S\colon S\to RK(S)=R(A)}$, the $S$-component of the adjunction unit $\eta$. As in the case of groups, we call the middle object of the split extension $R(A)$ as in
\[
	\xymatrix{
	0 \ar[r] & A \ar[d]_-{(\eta_S)_A} \ar[r]^-k & G \ar[d]_-{(\eta_S)_G} \ar@<.5ex>[r]^-f & B \ar@<.5ex>[l]^-s \ar@{=}[d]\ar[r] & 0\\
	0 \ar[r] & KR(A) \ar[r]_-\kappa & A\wr B \ar@<.5ex>[r]^-\pi & B \ar@<.5ex>[l]^-\sigma \ar[r] & 0
	}
\]
the \emph{wreath product} of $A$ and $B$ and denote it by $A\wr B$. It turns out that this morphism of extensions is always a monomorphism:

\begin{lemma}\label{Lemma Eta Mono}
	The unit $\eta$ of the adjunction $K\dashv R$ is a monomorphism.
\end{lemma}
\begin{proof}
	It is a well-known categorical fact that the components of the unit of an adjunction are monomorphisms if and only if the left adjoint is a faithful functor. In the case of $K\dashv R$, faithfulness amounts to the condition that whenever we have two morphisms $\alpha$ and $\beta$ of split extensions over $B$ as in
	\[
		\xymatrix{
		0 \ar[r] & A \ar@<-.5ex>[d]_-{\alpha_A} \ar@<.5ex>[d]^-{\beta_A} \ar[r]^-k & G \ar@<-.5ex>[d]_-{\alpha_G} \ar@<.5ex>[d]^-{\beta_G} \ar@<.5ex>[r]^-f & B \ar@<.5ex>[l]^-s \ar@{=}[d]\ar[r] & 0\\
		0 \ar[r] & C \ar[r]_-l & H \ar@<.5ex>[r]^-g & B \ar@<.5ex>[l]^-t \ar[r] & 0\text{,}
		}
	\]
	if $\alpha_A=\beta_A$, then $\alpha_G=\beta_G$. In any semi-abelian category this is indeed true, because protomodularity implies that the morphisms $k$ and $s$ are jointly epimorphic as we recalled in~\ref{Section Context}.
\end{proof}

Thus we proved, essentially without any effort:

\begin{theorem}\label{Theorem Embedding Split}
	In a semi-abelian category $\X$, for any objects $A$ and $B$ there exists a universal\footnote{Universality here means that the property depicted in~\eqref{Eq Diagram Universality} holds, which includes functorial dependence of $R(A)$ on $A$.} split extension
	\[
		R(A)=\bigl({0 \to KR(A) \to A\wr B \leftrightarrows B \to 0}\bigr)
	\]
	over $B$ into which each split extension
	\[
		S=\bigl({0 \to A \to G \leftrightarrows B \to 0}\bigr)
	\]
	embeds if and only if the category $\X$ is locally algebraically cartesian closed, in which case $K\dashv R$ for each chosen object $B$ and the embedding is given by the $S$-component $\eta_S\colon S\to RK(S)$ of the unit of this adjunction.
\end{theorem}
\begin{proof}
	It was explained in~\ref{Section Construction Wreath} how the existence of a right adjoint functor~$R$ yields a universal split extension. Conversely, all universal split extensions over an object~$B$ taken together conspire to a functor $R$. By the universal properties those split extensions satisfy, $R$ is right adjoint to $K$.
\end{proof}

In other words, a \emph{universal Kaluzhnin--Krasner embedding theorem for split extensions} exists for locally algebraically cartesian closed semi-abelian categories, and only for those. This means that within the semi-abelian context, we can only hope for the validity of a universal Kaluzhnin--Krasner embedding theorem (for all extensions) when the category is (LACC). Now as already mentioned above, such categories are scarce. This becomes especially concrete in the setting of non-associative algebras over a field, by which we mean any type of algebras over a field~$\K$ in the ordinary sense, where we have a $\K$-vector space equipped with a bilinear multiplication satisfying certain identities which need not include associativity. Such a category is called a \emph{variety of non-associative algebras over a field $\K$}. It is indeed known that over an infinite field no such variety can be (LACC), unless it is the variety of $\K$-Lie algebras~\cite{GM-VdL2,GM-VdL3}, from which we deduce an important consequence:

\begin{theorem}\label{Theorem Algebras}
	A variety of non-associative algebras over an infinite field admits a universal Kaluzhnin--Krasner embedding theorem for split extensions if and only if it is the variety of Lie algebras.\noproof
\end{theorem}

That is to say, there is no hope of ever extending the result of~\cite{PRS} to other types of algebras over a field, such as associative algebras, Jordan algebras or Leibniz algebras, in a way which keeps the universality of that embedding fully intact. (One might instead envision \emph{non-universal} Kaluzhnin--Krasner embeddings, however.) Other semi-abelian categories which are excluded because they are known not to be (LACC) are the categories of (commutative or non-commutative) loops, Heyting semilattices, digroups \cite[Examples 4.10]{acc}; and the categories of associative rings (commutative or not, boolean or not)~\cite{Gray2012}.

Leaving this potential obstacle aside, in the remainder of this article we focus on the positive side of Theorem~\ref{Theorem Embedding Split}, extending the general universal 	Kaluzhnin--Krasner embedding theorem from split extensions to arbitrary extensions.

\section{Arbitrary extensions: crude embedding}
Again fixing a semi-abelian category $\X$, we write $\Ext(B)$ for the category of extensions over $B$ in $\X$: a morphism $\phi=(\phi_A,\phi_G)$ from
\begin{equation}\label{Eq Extension E}
	E=(A,G,k,f)=\bigl(\xymatrix
	{0 \ar[r] & A \ar[r]^-k & G \ar[r]^-f & B \ar[r] & 0}\bigr)
\end{equation}
to $F=(C,H,l,g)$ satisfies $\phi_G\circ k=l\circ \phi_A$ and $g\circ \phi_G=f$.
\[
	\xymatrix{
	0 \ar[r] & A \ar[d]_-{\phi_A} \ar[r]^-k & G \ar[d]_-{\phi_G} \ar[r]^-f & B  \ar@{=}[d]\ar[r] & 0\\
	0 \ar[r] & C \ar[r]_-l & H \ar[r]_-g & B \ar[r] & 0
	}
\]
We write $U\colon \Ext(B)\to \X$ for the forgetful functor which sends an extension $E=(A,G,k,f)$ to the kernel $A$, so that $A=U(E)$.

\begin{remark}\label{Remark U Faithful}
	Note that, unlike the forgetful functor $K$, the functor $U$ is not faithful. A counterexample in the category of abelian groups is given by the morphisms of extensions in the diagram
	\[
		\xymatrix{
		0 \ar[r] & \Z \ar@{=}[d] \ar[r]^-{(1_\Z,0)} & \Z\times \Z \ar@<-.5ex>[d]_-{1_{\Z\times \Z}} \ar@<.5ex>[d]^-{\beta} \ar[r]^-{\pi_2} & \Z  \ar@{=}[d]\ar[r] & 0\\
		0 \ar[r] & \Z \ar[r]_-{(1_\Z,0)} & \Z\times \Z \ar[r]_-{\pi_2} & \Z \ar[r] & 0\text{,}
		}
	\]
	where $\beta(m,n)=(m+n,n)$.
\end{remark}

We first work towards a crude version of the Kaluzhnin--Krasner embedding theorem for arbitrary extensions, based on a simple reduction from the non-split to the split case. This depends on the adjunction which exists between extensions and split extensions over $B$.

\subsection{The split extension universally induced by an extension}\label{Subsec Split Extension As Algebra}
The forgetful functor $P\colon \ExtS(B)\to \Ext(B)$ which sends a split extension $(A,G,k,f,s)$ to the extension $(A,G,k,f)$ has a left adjoint $L\colon \Ext(B)\to \ExtS(B)$. It suffices to see that there is a natural bijection between the two types of situations in Figure~\ref{Figure P |- L}. The $E$-component $\lambda_E\colon E\to PL(E)$ of the adjunction unit $\lambda$ is induced by the coproduct inclusion ${\iota_1\colon G\to G+B}$.

\begin{figure}
	\[
		\left.\vcenter{\xymatrix{A \ar[d]_-k \ar[r]^-{\phi_H} & C \ar[d]^-l\\
		G \ar[r]^-{\phi_G} \ar[d]_-f & H \ar[d]^-g\\
		B \ar@{=}[r] & B}}\middle|
		\vcenter{\xymatrix{KL(E) \ar[d]_-{\ker\langle f, 1_B\rangle } \ar[r]^-{\widetilde{\phi}_H} & C \ar[d]^-l\\
		G+B \ar[r]^-{\langle\phi_G,t\rangle} \ar@<.5ex>[d]^-{\langle f, 1_B\rangle} & H \ar@<.5ex>[d]^-g\\
		B \ar@<.5ex>[u]^-{\iota_2}\ar@{=}[r] & B \ar@<.5ex>[u]^-t}}\right.
	\]
	\caption{Left, a morphism of extensions ${\phi\colon E\to P(S)}$; right, the corresponding morphism of split extensions $\widetilde{\phi}\colon{L(E)\to S}$.}\label{Figure P |- L}
\end{figure}

Note that a morphism of extensions $\phi\colon PL(E)\to E$ such that $\phi\comp \lambda_E=1_E$ is completely determined by the choice of a splitting $s\colon B\to G$ of the morphism~$f$: in the diagram
\[
	\xymatrix{
	0 \ar[r] & KL(E) \ar@{-->}[d] \ar[r] & G+B \ar@{-->}[d]_-{\langle?,s\rangle} \ar[r]^-{\langle f,1_B\rangle} & B  \ar@{=}[d]\ar[r] & 0\\
	0 \ar[r] & A \ar[r]_-k & G \ar[r]_-f & B \ar[r] & 0\text{,}
	}
\]
the condition $\phi\comp \lambda_E=1_E$ forces $?=1_G$. This means that a split extension is the same thing as an algebra for the pointed endofunctor
\[
	(PL\colon \Ext(B)\to \Ext(B),\; \lambda\colon 1_{\Ext(B)}\To PL)
\]
and the category of such algebras is isomorphic to $\ExtS(B)$.

\subsection{Towards a crude embedding theorem for extensions}
Adjunctions compose, and this provides us with a crude embedding theorem for extensions. We do indeed have that the composite $KL=UPL\colon \Ext(B)\to \X$ has a right adjoint $W\coloneq PR$, as in the diagram
\[
	\xymatrix@C=4em{\Ext(B) \ar@{<-}`d[r]`[rr]_-{W=PR}[rr] \ar@<1ex>[r]^-{L} \ar@{}[r]|-{\bot} & \ExtS(B) \ar@{}[r]|-{\bot} \ar@<1ex>[l]^-{P} \ar@<1ex>[r]^-{K=UP} & \X \ar@{<-}`u[l]`[ll]_-{KL=UPL}[ll] \text{.} \ar@<1ex>[l]^-{R}}
\]
For any object $A$ of $\X$, the extension $W(A)$ is just the wreath product
\[
	\xymatrix{
		0 \ar[r] & UPR(A) \ar[r]^-\kappa & A\wr B \ar[r]^-\pi & B  \ar[r] & 0
	}
\]
again, but now with the section $\sigma$ forgotten. Note that the left adjoint $KL=UPL$ does \emph{not} coincide with the forgetful functor $U$, which might be unexpected in view of the original Kaluzhnin--Krasner embedding theorem. Here the $E$-component $\upsilon_E$ of the unit $\upsilon$ of the adjunction takes the form ${\upsilon_E\colon E\to WUPL(E)}=WKL(E)$. Just as in the case of split extensions (Lemma~\ref{Lemma Eta Mono}), we may prove that this natural transformation is always a monomorphism.

\begin{lemma}
	The unit $\upsilon$ of the adjunction $KL\dashv W$ is a monomorphism.
\end{lemma}
\begin{proof}
	We already know that $K$ is faithful by Lemma~\ref{Lemma Eta Mono}. Since the unit $\lambda$ of the adjunction $L\dashv P$ is a monomorphism by its construction as a coproduct inclusion, it follows that $L$ is faithful as well. Therefore the composite $KL$ is faithful, and hence the unit of the adjunction $KL\dashv PR=W$ is a monomorphism.
\end{proof}

We find:

\begin{theorem}\label{Theorem Embedding Crude}
	In a locally algebraically cartesian closed semi-abelian category $\X$, for any objects $X$ and $B$ there exists a universal extension
	\[
		W(X)=\bigl({0 \to KR(X) \to X\wr B \to B \to 0}\bigr)
	\]
	over $B$ into which each extension
	\[
		E=\bigl({0 \to A \to G \to B \to 0}\bigr)
	\]
	such that $X=KL(E)$ embeds. For a given extension $E$, the embedding is given by the $E$-component ${\upsilon_E\colon E\to WKL(E)}$ of the unit of the adjunction $KL\dashv W$.\noproof
\end{theorem}

Our aim is now to deduce from this result an embedding theorem which is closer to the original one for groups: it will, for instance, take into account the set-theoretical splittings an extension may have. For each extension this involves the construction of a non-canonical map, one which cannot be deduced from the adjointness coming from local algebraic cartesian closedness.

\section{Embedding into the wreath product \texorpdfstring{$A\wr B$}{AwrB}}
We would like to be able to embed an extension $E$ from $A=U(E)$ to $B$ into the wreath product $W(A)$ rather than into $WKL(E)$. In other words, we require the existence of a monomorphism $\phi\colon E\to WU(E)$ for each $E$. Let us analyze this situation in detail.

\subsection{On the existence of \texorpdfstring{$\phi\colon E\to W(A)$}{phi}}
Assuming that a morphism of extensions $\phi\colon E\to W(A)$ does indeed exist, by universality of $WKL(E)$ we obtain a unique morphism $\chi\colon KL(E)\to U(E)$ in $\X$ as in
\[
	\xymatrix{E \ar[r]^-{\upsilon_E} \ar[rd]_-{\forall\phi} & WKL(E) \ar@{-->}[d]^-{W(\chi)} & KL(E)\ar@{-->}[d]^-{\exists !\chi}\\
	& WU(E) & U(E)\text{.}}
\]

Theorem~\ref{Theorem Embedding Split} implies that such a $\phi$ exists when $E=P(S)$ is a split extension: then we may take $\phi=P(\eta_S)$. In this case, the induced morphism $\chi$ is $U$ applied to the $(PL,\lambda)$-algebra structure of $P(S)$---see \ref{Subsec Split Extension As Algebra}---which is $P$ of the counit ${\epsilon_{S}\colon LP(S)\to S}$ of $L\dashv P$ at $S$. Indeed, the composite $WK(\epsilon_S)\circ \upsilon_E$ is equal to~$P(\eta_S)$---since \(\upsilon_E = P\bigl(\eta_{L(E)}\bigr) \circ \lambda_E\), by naturality of $\eta$ and by the triangular identity for $L\dashv P$ as in the commutative diagram
\[
	\xymatrix@=4em{ P(S) \ar@{=}[d] \ar[r]^-{\lambda_{P(S)}} & PLP(S) \ar[d]_-{P(\epsilon_S)} \ar[r]^-{P(\eta_{LP(S)})} & PRKLP(S)\ar[r]^-{WK(\epsilon_S)} \ar[d]^-{PRK(\epsilon_S)} & WK(S) \ar@{=}[d] \\
	E \ar@{=}[r] & P(S) \ar[r]_-{P(\eta_S)} & PRK(S) \ar@{=}[r] & WU(E)\text{.}}
\]
Conversely, if $\chi\colon KL(E)\to U(E)$ happens to be induced by a morphism of extensions ${PL(E)\to E}$, then $E$ carries a $PL$-algebra structure---an arrow \({PL(E) \to E}\), which \emph{a priori} need not be compatible with the unit $\lambda$, but is enough to imply that~$E$ was a split extension in the first place.

This means that we cannot hope that $\chi$ is $U(\underline{\phi})$ for some morphism of extensions $\underline{\phi}\colon PL(E)\to E$. As a consequence, its existence (as a morphism of $\X$ which does not underlie an extension) must follow from a construction outside of the realm of the adjoint functors which we have been considering so far. The non-canonicity of these maps forces us to work on a case-by-case basis. This means understanding the structure of the kernel $KL(E)$ in
\[
	PL(E)=\bigl(\xymatrix
	{0 \ar[r] & KL(E) \ar[r] & G+B \ar[r]^-{\langle f,1_B\rangle} & B \ar[r] & 0}\bigr)
\]
for any given extension
\begin{equation*}\tag{\ref{Eq Extension E}}
	E=\bigl(\xymatrix
	{0 \ar[r] & A \ar[r]^-k & G \ar[r]^-f & B \ar[r] & 0}\bigr)
\end{equation*}
in some concrete (LACC) semi-abelian category. This becomes feasible when the objects in the category have underlying sets---for instance, when we work in a semi-abelian variety of algebras. In our examples, the morphism $\phi$ will then be induced by a set-theoretical splitting $s$ of $f$, which may satisfy further properties such as linearity in the case of Lie algebras. In any variety of algebras, when ${\phi\colon E\to WU(E)}$ is a monomorphism as desired, such a section $s$ of~$f$ can only be compatible with the section~$\sigma$ of the wreath product split extension (i.e., we can only have \(\phi_G \circ s = \sigma\)) if $s$ is a morphism, so that $E$ is a split extension.

In what follows, we shall explain in detail how, in the cases of groups and Lie algebras, this approach leads to the known results. As we shall see, this is far from trivial. We make this effort with the aim of sketching a procedure which can in principle be mimicked in other settings to obtain new results.

\subsection{The case of groups}
In the category of groups, let us consider an extension~$E$ as in~\eqref{Eq Extension E} above, in order to describe the structure of the induced group $KL(E)$, which will provide us with a group monomorphism $\phi\colon E\to WU(E)$.

\begin{lemma}
	The kernel $\Ker (\langle f,1_B \rangle)$ of the induced arrow $\langle f,1_B \rangle \colon G+B \to B$ admits the presentation $P = \langle S\mid R \rangle$ with $S = G \times B$ and
	\[
		R = \{(1,b) = 1 \mid b \in B\} \cup \{(g,b)(g',bf(g)) = (gg',b) \mid \text{$g$, $g' \in G$, $b \in B$}\}\text{.}
	\]
	The idea is to see an element $(g,b)$ of $S$ as the word $bg(bf(g))^{-1}$ in $G+B$.
\end{lemma}
\begin{proof}
	Let \(P\) denote a group admitting the presentation of the statement. We will construct an isomorphism between \(P\) and \(\Ker(\langle f,1_B \rangle)\).
	
	First, according to the idea given in the statement of the lemma, let us define \(\varphi \colon P \to G+B\) sending a generator \((g,b)\) of \(P\), where \(g \in G\), \(b \in B\), to the element \(bg(bf(g))^{-1}\) of \(G+B\). Since these elements verify the relations of \(R\), this assignment forms a well-defined group homomorphism from \(P\) to \(G+B\). Moreover, for all \(g \in G\) and all \(b \in B\),
	\[
		\langle f,1_B \rangle \bigl(\varphi(g,b)\bigr) = \langle f,1_B \rangle \bigl(bg(bf(g))^{-1}\bigr) = bf(g)(bf(g))^{-1} = 1
	\]
	so \(\varphi\) corestricts to the kernel of \(\langle f,1_B \rangle\) to give us a morphism \(\phi \colon P \to \Ker(\langle f,1_B \rangle)\).
	
	Conversely, any \(h \in \Ker(\langle f,1_B \rangle) \leq G+B\) can be uniquely written in reduced form \(h = b_1 g_1 \cdots b_n g_n\), i.e.\ with \(b_i \in B \setminus \{1\}\) for \(i \in \{2, \dots , n\}\), \(g_i \in G \setminus \{1\}\) for \(i \in \{1, \dots , n-1\}\), \(b_1 \in B\) and \(g_n \in G\). Then we define \(\psi \colon \Ker(\langle f,1_B \rangle) \to P\) by setting
	\[
		\psi(h) \coloneq (g_1,b_1) \bigl(g_2,b_1 f(g_1) b_2\bigr) \cdots \bigl(g_n,b_1 f(g_1) \cdots b_{n-1} f(g_{n-1}) b_n\bigr)
	\]
	for \(h = b_1 g_1 \cdots b_n g_n \in \Ker(\langle f,1_B \rangle)\) in reduced form. This function is well defined since the reduced form is unique. At this point it is not so clear whether or not it is a group homomorphism, but this is not needed for our purposes. It will actually follow once we prove that $\psi$ is the inverse function of the group homomorphism $\phi$.
	
	So, let us check that \(\phi\) and \(\psi\) are each other's inverse. For an element \(h = b_1 g_1 \cdots b_n g_n\) of \(\Ker(\langle f,1_B \rangle)\) written in reduced form, we compute
	\begin{align*}
		\phi\bigl(\psi(h)\bigr)
		= & \; \phi\bigl((g_1,b_1) \bigl(g_2,b_1 f(g_1) b_2\bigr) \cdots \bigl(g_n,b_1 f(g_1) \cdots b_{n-1} f(g_{n-1}) b_n\bigr)\bigr) \\
		= & \; \bigl(b_1 g_1 (b_1 f(g_1))^{-1}\bigr) \bigl(b_1 f(g_1) b_2 g_2 (b_1 f(g_1) b_2 f(g_2))^{-1}\bigr)\cdots                  \\
		  & \cdot \bigl(b_1 f(g_1) \cdots b_{n-1} f(g_{n-1}) b_n g_n (b_1 f(g_1) \cdots b_{n-1} f(g_{n-1}) b_n f(g_n))^{-1}\bigr)       \\
		= & \; b_1 g_1 \cdots b_n g_n = h
	\end{align*}
	as wanted, the last step using that \(b_1 f(g_1) \cdots b_{n-1} f(g_{n-1}) b_n f(g_n) = 1\) since our element \(h = b_1 g_1 \cdots b_n g_n\) is in the kernel of~\(\langle f,1_B \rangle\). For the other direction, take \(p = (g_1,b_1) \cdots (g_n,b_n) \in P\) written in a minimal way (which is possible since the relations of \(R\) reduce the length of the words). We compute
	\begin{align*}
		\phi(p) & = \bigl(b_1 g_1 (b_1 f(g_1))^{-1}\bigr) \cdots \bigl(b_n g_n (b_n f(g_n))^{-1}\bigr)                                                     \\
		        & = b_1 g_1 \bigl((b_1 f(g_1))^{-1} b_2\bigr) g_2 \cdots g_{n-1} \bigl((b_{n-1} f(g_{n-1}))^{-1} b_n\bigr) g_n (b_n f(g_n))^{-1} 1\text{.}
	\end{align*}
	Then, since we chose a minimal way of writing \(p\), we have two possibilities for the reduced form of \(\phi(p)\). If \(b_n f(g_n) \neq 1\), the rewriting of \(\phi(p)\) above is its reduced form and, by definition of \(\psi\), we have
	\begin{align*}
		\psi\bigl(\phi(p)\bigr) & = (g_1,b_1) \bigl(g_2,b_1 f(g_1) ((b_1 f(g_1))^{-1} b_2)\bigr)\cdots                                                         \\
		                        & \quad\cdot \bigl(g_n,b_1 f(g_1) \cdots  ((b_{n-2} f(g_{n-2}))^{-1} b_{n-1}) f(g_{n-1}) ((b_{n-1} f(g_{n-1}))^{-1} b_n)\bigr) \\
		                        & \quad\quad\cdot \bigl(1,b_1 f(g_1) \cdots  ((b_{n-1} f(g_{n-1}))^{-1} b_n) f(g_n) (b_n f(g_n))^{-1}\bigr)                    \\
		                        & = (g_1,b_1) (g_2,b_2) \cdots (g_n,b_n) (1,1) = p\text{.}
	\end{align*}
	In the other case, the reduced form of \(\phi(p)\) is given by
	\[
		b_1 g_1 \bigl((b_1 f(g_1))^{-1} b_2\bigr) g_2 \cdots \bigl((b_{n-2} f(g_{n-2}))^{-1} b_{n-1}\bigr) g_{n-1} \bigl((b_{n-1} f(g_{n-1}))^{-1} b_n\bigr) g_n
	\]
	and we also have
	\begin{align*}
		\psi\bigl(\phi(p)\bigr) & = (g_1,b_1) \bigl(g_2,b_1 f(g_1) ((b_1 f(g_1))^{-1} b_2)\bigr)\cdots                                                         \\
		                        & \quad \cdot \bigl(g_n,b_1 f(g_1) \cdots ((b_{n-2} f(g_{n-2}))^{-1} b_{n-1}) f(g_{n-1}) ((b_{n-1} f(g_{n-1}))^{-1} b_n)\bigr) \\
		                        & = (g_1,b_1) (g_2,b_2) \cdots (g_n,b_n) = p\text{,}
	\end{align*}
	finishing the proof.
\end{proof}

If now $s\colon B\to G$ is a set-theoretical splitting of $f$, then we may consider the group homomorphism
\[
	\chi\colon KL(E)\to U(E)\colon (g,b)\mapsto s(b)\cdot g\cdot s(b\cdot f(g))^{-1}
\]
which is well defined because $(1,b)$ is sent to $s(b)\cdot 1\cdot s(b\cdot f(1))^{-1}=1$ while for $(g,b)(g',bf(g))$ we have
\begin{align*}
	\chi(g,b) \cdot \chi(g',bf(g)) & = \bigl(s(b) \cdot g \cdot s(b \cdot f(g))^{-1}\bigr) \cdot \bigl(s(bf(g)) \cdot g' \cdot s(bf(g) \cdot f(g'))^{-1}\bigr) \\
	                               & =s(b) \cdot gg' \cdot s(b \cdot f(gg'))^{-1}
\end{align*}
which is the image of $(gg',b) = (g,b)(g',bf(g))$.

The corresponding morphism $\phi\colon E\to WU(E)$ is the composite
\[
	W(\chi)\comp \upsilon_E=W(\chi)\comp P(\eta_{L(E)})\comp \lambda_E\text{.}
\]
Its $G$-component $\phi_G$ sends $g\in G$ to $g\in G+B$ to
\[
	(h_g,f(g))\in  KL(E)\wr B=\Set(B, KL(E)) \rtimes B
\]
where $h_g\colon B\to KL(E)\colon b\mapsto (g,b)$ as explained in the paragraph immediately below Theorem~\ref{Theorem KK}. In accordance with \eqref{Eq R Natural}, this couple $(h_g,f(g))$ in $KL(E)\wr B$ is in turn sent to $(\chi\comp h_g,f(g))\in A\wr B$. Note that $(\chi\comp h_g)(b)=s(b)\cdot g\cdot s(b\cdot f(g))^{-1}$, so that we regain the classical formula~\cite{KKThm} for the Kaluzhnin--Krasner embedding.

\subsection{The case of Lie algebras}\label{Lie case}
Given a field $\K$ and a split extension of $\K$-Lie algebras~$S$ as in~\eqref{Eq Split Extension S}, it is known~\cite{GrayLie} that $KR(A)$ is isomorphic to $\Vect_{\K}(\overline{B}, A)$, where $\overline{B}$ denotes the universal enveloping algebra of~$B$.

Given an extension~$E$ as in~\eqref{Eq Extension E}, the canonical embedding $E\to RKL(E)$ restricted to $A \to KRKL(E)$ sends any $a \in A$ to
\[
	h_a \colon \overline{B} \to KL(E)\colon b_1 \cdots b_r \mapsto b_1(b_2(\cdots (b_r a) \cdots ))\text{.}
\]
Let $s$ be a linear splitting of $f$. Then $h_a$ induces
\[
	h'_a \colon \overline{B} \to A\colon  b_1 \cdots b_r \mapsto s(b_1)(s(b_2)(\cdots (s(b_r) a) \cdots ))\text{.}
\]
To check that this morphism is well defined, some computations need to be made; here we may imitate~\cite[Section~3]{PRS}. Note that choosing a linear section is the same as choosing a basis that complements $A$ in $G$.

Hence we obtain a morphism $A \to UWU(E)=KRU(E) = \Vect_{\K}(\overline{B}, A)$, which in turn induces the morphism of extensions $\phi \colon E \to WU(E)$, whose $G$-component sends $g \in G$ to
\[
	(h'_{g-sf(g)}, f(g) )\in WU(E) \wr B = \Vect_{\K}(\overline{B}, A) \rtimes B\text{.}
\]
Thus we recover the Kaluzhnin--Krasner embedding from~\cite{PRS}.

\subsection{Further examples}
From the above it is clear that, even though a universal Kaluzhnin--Krasner embedding in its standard form does not follow right away from Theorem~\ref{Theorem Embedding Crude}, there is very little hope of establishing such an embedding in contexts where that theorem is not also valid. As a consequence, the category of crossed modules, as well as the examples worked out in~\cite{GM-G}, being (LACC) semi-abelian categories, all satisfy Theorem~\ref{Theorem Embedding Crude}, hence are good candidates for a ``classical'' Kaluzhnin--Krasner embedding. We will, however, end the article with a slightly different result: a universal embedding theorem for \emph{abelian} split extensions which holds in \emph{any} semi-abelian variety of algebras.

\section{The case of abelian actions}
We pick a semi-abelian variety of algebras $\X$. Note that these are exactly the protomodular varieties, which were characterized in~\cite{Borceux-Bourn,Bourn-Janelidze}.
We are now going to explain that the condition \LACC\ is not necessary in such a category, if we want to embed just the \emph{abelian} actions, which correspond to split extensions equipped with a Beck module structure~\cite{Beck}.

Given an object $B$, recall that a \defn{Beck module over $B$} is an extension \eqref{Eq Extension E} that carries an internal abelian group structure in $\Ext(B)$---determined, in particular, by a unit $s\colon 1_B\to f$ and a multiplication $m\colon f\times_Bf\to f$ in the category $\X/_B$ of objects over $B$---which automatically makes this extension \eqref{Eq Extension E} a split extension. Let us write $\Ab(\Ext(B))$ for the category of (split) extensions over~$B$, equipped with an abelian group structure; its morphisms are morphisms of split extensions. This category is abelian, and since $\X$ is a variety of algebras, \cite[Theorem 2.9]{Gray2012} implies that the lifting $\underline{K}\colon{\Ab(\Ext(B))\to \Ab(\X)}$ of the functor $K\colon {\ExtS(B)\to \X}$ to abelian group objects has a right adjoint~$\underline{R}$. From this we deduce:

\begin{theorem}\label{Theorem Abelian Embedding Split}
	In a semi-abelian variety $\X$, for any object $B$ and any abelian group object $A$ there exists a universal Beck module
	\[
		\underline{R}(A)=\bigl({0 \to \underline{K}\underline{R}(A) \to A\mathbin{\underline{\wr}} B \leftrightarrows B \to 0}\bigr)
	\]
	over $B$ into which each abelian split extension of the form
	\[
		S=\bigl({0 \to A \to G \leftrightarrows B \to 0}\bigr)
	\]
	embeds. This embedding is given by the $S$-component $\underline{\eta}_S\colon S\to \underline{R}\underline{K}(S)$ of the unit of the adjunction $\underline{K}\dashv \underline{R}$.\noproof
\end{theorem}

Note that we underlined the wreath product symbol to distinguish $A\mathbin{\underline{\wr}} B$ from the ordinary wreath product $A\wr B$.

\begin{remark}
	To show that the unit $\underline{\eta}$ of the adjunction $\underline{K}\dashv \underline{R}$ is a monomorphism, the reasoning of Lemma~\ref{Lemma Eta Mono} applies.
\end{remark}

\begin{remark}
	It is well known that the category $\Ab(\Ext(B))$ is again a variety of algebras, and since it is an abelian category as well, it is a category of modules over a ring~\cite[Exercise~4.F]{Freyd}. The ring $\Lambda$ in question is the endomorphism ring of the free Beck module (over $B$) with a single generator. In particular, write $\Mod_\Pi\simeq\Ab(\Ext(0))$; the morphism $0\to B$ induces a ring map $\Pi\to \Lambda$. The functor $\underline{K}\colon\Mod_\Lambda \to \Mod_\Pi$ becomes restriction of scalars, so that its right adjoint is $A\mapsto \Mod_\Pi(\Lambda,A)$.
\end{remark}

By the analysis made in~\cite{Bourn-Janelidze:Torsors}, the situation simplifies when the category $\X$ satisfies a mild additional condition, much weaker than (LACC), called the \defn{Smith is Huq condition} in~\cite{MFVdL}: then a Beck $B$-module structure on an internal abelian group object $A$ is completely determined by a split extension from $A$ to $B$. So the abelian split extensions are precisely the split extensions with an abelian kernel. For instance, for a field $\K$, any variety of $\K$-algebras satisfies this condition. Further examples of such categories are given in~\cite{acc,Rodelo:Moore}. Note that in this setting, an abelian object is a $\K$-vector space equipped with the trivial multiplication, so that a Beck module over an algebra $B$ is a split extension such as $S$ above where the result of multiplying two elements of~$A$ is always zero.

For instance, in the case of $\K$-Lie algebras, we have $\Pi=\K$ and $\Lambda=\overline{B}$, the universal enveloping algebra of the Lie algebra $B$. Hence $\underline{K}\colon \Mod_{\overline{B}}\to \Vect_\K$ is the forgetful functor, and its right adjoint takes a $\K$-vector space $A$ to the space $\Vect_\K(\overline{B},A)$ with its canonical $\overline{B}$-module structure as in~\ref{Lie case}.

\section*{Acknowledgements}
The authors would like to express their gratitude to the organisers of the Group Theory Seminar at ICMAT and, in particular, to Henrique A.~Mendes da Silva e Souza, for asking a question that led to this research.


\providecommand{\noopsort}[1]{}
\providecommand{\bysame}{\leavevmode\hbox to3em{\hrulefill}\thinspace}
\providecommand{\MR}{\relax\ifhmode\unskip\space\fi MR }
\providecommand{\MRhref}[2]{%
	\href{http://www.ams.org/mathscinet-getitem?mr=#1}{#2}
}
\providecommand{\href}[2]{#2}

\end{document}